\documentclass[12pt]{article}
\usepackage{amsmath,amsfonts, amsthm ,graphicx, calc}
\newcommand{\R}{\mathbb{R}}
\newcommand{\dd}{f}
\newcommand{\Z}{\mathbb{Z}}
\newcommand{\N}{\mathbb{N}}
\newcommand{\one}{\mathbf{1}}
\newcommand{\lang}{\Big\langle}
\newcommand{\rang}{\Big\rangle}
\newcommand{\vol}{\text{vol}}
\newtheorem{thm}{Theorem}

\newtheorem{lem}[thm]{Lemma}

\newtheorem{prop}[thm]{Proposition}
\newtheorem{defi}[thm]{Definition}

\def\vol{\mathrm{vol}}

\title{Enumeration formulas for Young tableaux in a diagonal strip}

\author{Yuliy Baryshnikov\thanks{Mathematical and Algorithmic Sciences, Bell Laboratories, Murray Hill, NJ, USA. email: \texttt{ymb@research.bell-labs.com} .}\and
Dan Romik\thanks{Einstein Institute of Mathematics, The Hebrew University, Jerusalem, Israel. email: \texttt{romik@math.huji.ac.il} .}
}

\begin{document}

\maketitle

\abstract{We derive combinatorial identities, involving the Bernoulli and Euler
numbers, for the numbers of standard Young tableaux of certain skew shapes. This
generalizes the classical formulas of D. Andr\'e on the number of
up-down permutations. The analysis uses a transfer operator approach
extending the method of Elkies, combined with an identity expressing the volume of a certain polytope in terms of a Schur function.  }

\vspace{130.0pt}


\newpage

\section{Introduction}

\subsection{Up-down permutations and the Euler, tangent and Bernoulli numbers}

An up-down permutation on $n$ elements is a permutation $\sigma \in S_n$ satisfying
$$ \sigma(1) < \sigma(2) > \sigma(3) < \sigma(4) > \cdots .$$ Up-down
permutations, also known as zig-zag- or alternating permutations, were
first studied in 1879 by D. Andr\'e \cite{andre1, andre2}. He showed
that if $A_n$ denotes the number of $n$-element up-down permutations,
then the exponential generating function
$$ f_{\text{up-down}}(x) = \sum_{n=0}^\infty \frac{A_n x^n}{n!} $$
is given by
\begin{equation}\label{eq:andreformula}
 f_{\text{up-down}}(x) = \tan x + \sec x.
 \end{equation}
In other words, $\tan x$ is the exponential g.f. of the even-indexed
$A_n$'s, and $\sec x$ is the e.g.f. of the odd-indexed $A_n$'s. This
relates the sequence $(A_n)$ to the  \emph{Euler numbers}
$E_n$ and \emph{tangent numbers} $T_n$, traditionally defined by the
Taylor series expansions
\begin{eqnarray*}
\sec x &=& \sum_{n=0}^\infty  \frac{(-1)^nE_{2n} x^{2n}}{(2n)!}, \\
\tan x &=& \sum_{n=1}^\infty \frac{T_n x^{2n-1}}{(2n-1)!}.
\end{eqnarray*}
In terms of these numbers, we have that
\begin{equation}\label{eq:updowneulertan}
 A_{2n} = (-1)^n E_{2n}, \qquad A_{2n-1} = T_n. \end{equation}
Recall also that the
tangent numbers are related to the \emph{Bernoulli
numbers} $B_n$, defined by the Taylor series expansion
$$ \frac{x}{e^x-1} = \sum_{n=0}^\infty \frac{B_n x^n}{n!},
$$
via the relation
\begin{equation} \label{eq:viarelation}
T_n = \frac{(-1)^{n-1} 4^n (4^n-1)}{2n} B_{2n}.
\end{equation}
For an amusing
appearance of Euler and tangent numbers unrelated to up-down
permutations, see \cite{borweinetal}. The notation $E_n, T_n$ and
$B_n$ will be used throughout the paper, always signifying the Euler,
tangent and Bernoulli numbers, respectively, and $A_n$ will be used
throughout to denote the number of up-down permutations of order $n$,
given by \eqref{eq:updowneulertan}.

\subsection{Standard Young tableaux}

It is well-known (\cite[Ex.\ 23, p. 68]{knuthvol3}, \cite[Ex.\ 7.64.a,
p. 469--470, 520]{stanleyvol2}), that up-down permutations can be
thought of as a special case of a \emph{standard Young tableau}. Recall that an \emph{integer partition}
is a sequence $\lambda=(\lambda_1,\lambda_2,\ldots,\lambda_k)$, where $\lambda_1\ge\lambda_2>\ldots>\lambda_k>0$ are integers. We identify such a partition $\lambda$ with its \emph{Young diagram}, which is the set $\{ (i,j)\in \N^2 : 1\le i\le k,\ 1\le j\le \lambda_i \}$, graphically depicted as a set of squares, (also called \emph{cells} or \emph{boxes}) in the plane, traditionally in the ``English notation'' whereby $y$-coordinate increases from top to bottom (similarly to matrix row indices in linear algebra). A \emph{skew Young diagram} is the difference $\lambda\setminus \mu$ of two Young diagrams where $\mu\subset \lambda$. If $\lambda\setminus \mu$ is a skew Young diagram, a \emph{standard Young tableau} (SYT) of shape $\lambda\setminus \mu$ is a filling of the boxes of $\lambda\setminus \mu$ with the integers $1,2,\ldots,|\lambda\setminus \mu|$ that is increasing along rows and columns. See Figure \ref{fig1} for an example.

\begin{figure}[h!] 
\begin{center}
\begin{tabular}{cc}
\begin{picture}(100,100)(0,0)
\multiput(10,10)(17,0){2}{\framebox(17,17)}
\multiput(27,27)(17,0){2}{\framebox(17,17)}
\multiput(27,44)(17,0){4}{\framebox(17,17)}
\multiput(44,61)(17,0){3}{\framebox(17,17)}
\multiput(44,78)(17,0){3}{\framebox(17,17)}
\end{picture}
& 
\begin{picture}(100,100)(0,0)
\multiput(10,10)(17,0){2}{\framebox(17,17)}
\multiput(27,27)(17,0){2}{\framebox(17,17)}
\multiput(27,44)(17,0){4}{\framebox(17,17)}
\multiput(44,61)(17,0){3}{\framebox(17,17)}
\multiput(44,78)(17,0){3}{\framebox(17,17)}
\put(15,15){$5$}
\put(30,15){$11$}
\put(32,32){$6$}
\put(49,32){$8$}
\put(32,49){$1$}
\put(49,49){$4$}
\put(64,49){$12$}
\put(81,49){$14$}
\put(49,66){$3$}
\put(64,66){$10$}
\put(81,66){$13$}
\put(49,83){$2$}
\put(66,83){$7$}
\put(83,83){$9$}
\end{picture}  \\
 (a) & (b)
\end{tabular}
\caption{(a) The skew Young diagram $(5,5,5,3,2)\setminus (2,2,1,1,0)$. (b) A~standard Young tableau (SYT). \label{fig1}}
\end{center}
\end{figure}
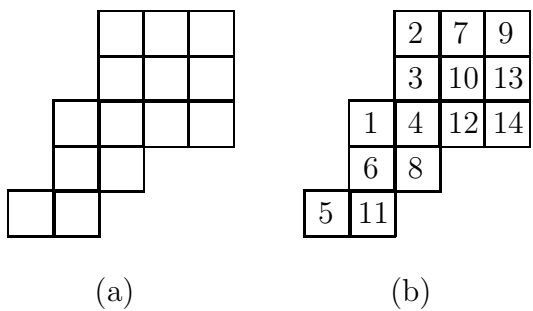

Up-down permutations on $n$ elements are in simple bijection with SYT's of shape
$$ (m+1,m,m-1,m-2,\ldots,3,2)\setminus (m-1,m-2,\ldots,1,0)$$
when $n=2m$ is even, or
$$ (m,m,m-1,m-2,\ldots,3,2)\setminus (m-1,m-2,\ldots,1,0)$$
when $n=2m-1$ is odd.
The bijection converts such an SYT into
an up-down permutation by reading the values starting from the
bottom-left box and going alternately to the right and up.
See Figure \ref{fig2}.

\begin{figure}[h!] 
\begin{center}
\begin{tabular}{cc}
\begin{picture}(180,70)(0,45)
\multiput(10,10)(17,17){5}{\framebox(17,17)}
\multiput(27,10)(17,17){5}{\framebox(17,17)}
\put(15,15){$5$}
\put(32,32){$2$}
\put(49,49){$4$}
\put(66,66){$3$}
\put(83,83){$1$}
\put(32,15){$8$}
\put(49,32){$9$}
\put(63,49){$10$}
\put(83,66){$7$}
\put(100,83){$6$}
\end{picture}  \!\!\!\!\!\!\!\!\!\!\!\!\!\!\!\!\!\!\!\!\!\!\!\!\!\!\!\!\!\!\!\!\!\!\!\!\!\!
& $\sigma = \left( \begin{array}{cccccccccc}
 1 &2 &3 &4 &5 &6 &7 &8 &9 &10 \\
 5 & 8 & 2 & 9 & 4 & 10 & 3 & 7 & 1 & 6
 \end{array}\right)$
\end{tabular}
\bigskip\bigskip
\caption{A standard Young tableau of shape $(5,5,4,3,2,1)\setminus (4,3,2,1,0,0)$ and the associated
up-down permutation. \label{fig2}}
\end{center}
\end{figure}
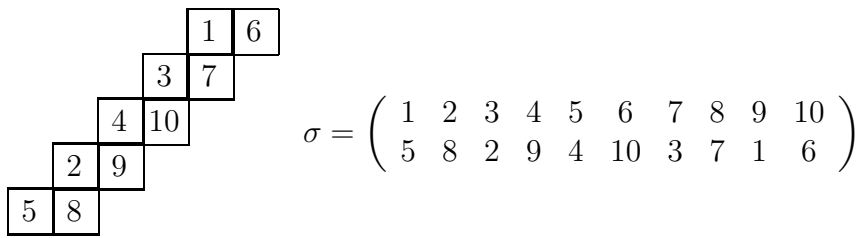

Our main interest will be in enumerating SYT's of
shapes in a special family. For a skew Young diagram $D=\lambda\setminus\mu$, denote by $\dd(D)$ the number of
SYT's of shape $D$ (this is frequently denoted by $f^{\lambda\setminus \mu}$, but the current notation will be more convenient for our purposes). Recall the formula of Aitken \cite{aitken} (later rediscovered by Feit \cite{feit}; see also \cite[Corollary 7.16.3, p. 344]{stanleyvol2}, \cite[Ex. 19, p. 67, 609]{knuthvol3}) for $\dd(D)$:
\begin{equation}
\label{eq:aitkeneq} 
\dd(D)=|\lambda\setminus \mu|! \det\Big( \frac{1}{(\lambda_i-i-\mu_j+j)!} \Big)_{i,j}.
\end{equation}
Using \eqref{eq:aitkeneq} can become difficult when the diagrams $\lambda$ and $\mu$ are large (even when their difference is small!), as the order of the determinant is equal to the number of parts in $\lambda$. As we shall see below, in certain cases one can give more concise formulas that also relate the numbers of SYT's of different shapes to each other, and notably to the ``zig-zag'' numbers $(A_n)_{n\ge 0}$.

\subsection{Standard Young tableaux in a strip}

The standard Young tableaux that we will consider will be of the following general shape, which we call an 
\emph{$m$-strip (diagonal) diagram}.

\begin{figure}[h!]
\begin{center}
{\includegraphics[height=6cm]{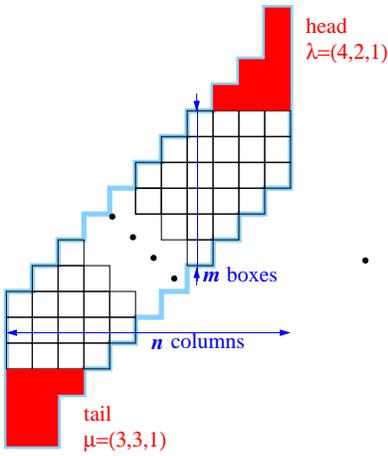}} 
\caption{An $m$-strip diagram. The head and tail rotated Young diagrams are read off in columns. \label{figmstrip}}
\end{center}
\end{figure}

The \emph{body} of the diagram consists of $n$ columns of $m$ boxes each. The \emph{head} and the \emph{tail}  are Young diagrams that are rotated and connected to the body by leaning against the sides of the body, as shown in the figure. We will usually think of $m$ as fixed and $n$ growing, which is why we use the term $m$-strip.

Our main result, Theorem \ref{generalthm} below, is a new formula for enumerating
$m$-strip tableaux (that is, SYT's whose shape is an $m$-strip diagram) in  terms of the zig-zag numbers $(A_n)_{n\ge 0}$. To introduce this formula, we start with some explicit formulas for small values of $m$, and then present the general formula that contains all these formulas as special cases.

\begin{thm}[$3$-strip tableaux] \label{thm3strip}
\begin{eqnarray}
\dd\left(
\raisebox{-3.1ex}{\includegraphics[width=1.4cm]{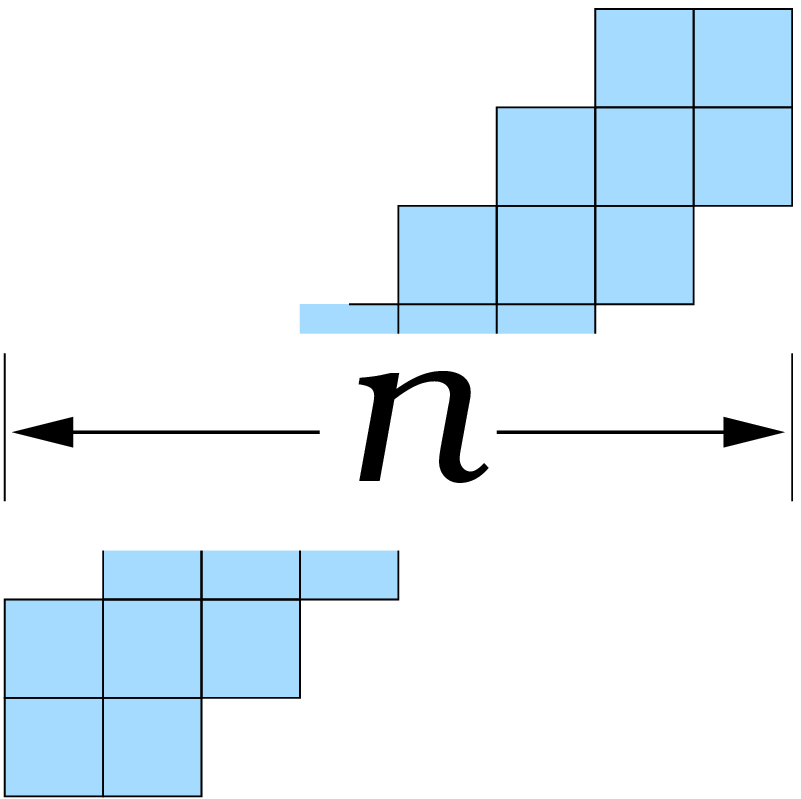}}
\right)
&=& 
\frac{(3n-2)!T_n}{(2n-1)!2^{2n-2}},
 \label{eq:form3strip1} \\
\dd\left(
\raisebox{-3.4ex}{\includegraphics[width=1.4cm]{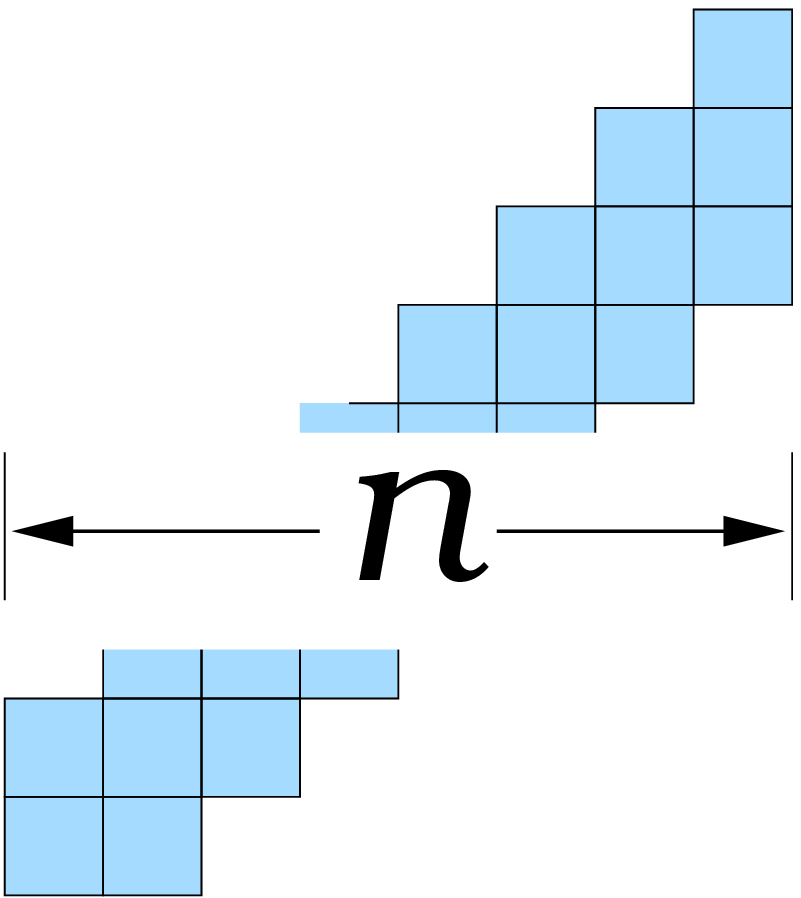}}
\right) 
&=& 
\frac{(3n-1)!T_n}{(2n-1)!2^{2n-1}}, \label{eq:form3strip2} \\
\dd\left(\raisebox{-3.7ex}{\includegraphics[width=1.4cm]{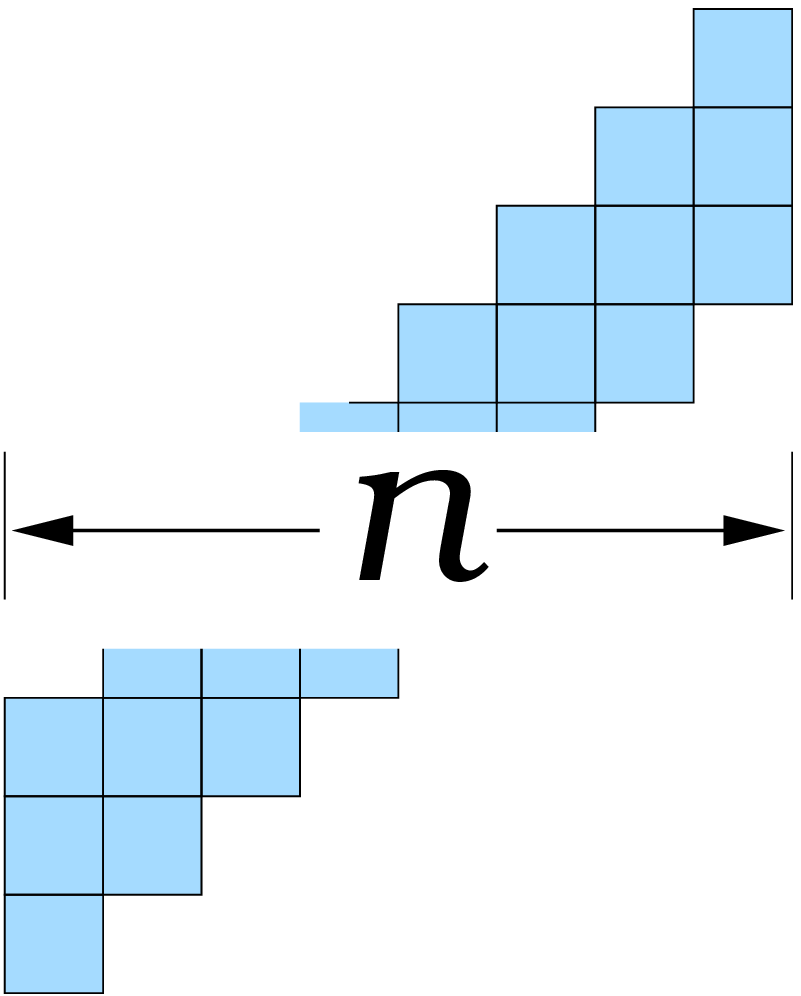}}\right) 
&=& 
\frac{(3n)!(2^{2n-1}-1)T_n}{(2n-1)!2^{2n-1}(2^{2n}-1)}. \label{eq:form3strip3}
\end{eqnarray}
\end{thm}

\begin{thm}[$4$-strip tableaux] \label{thm4strip}
\begin{eqnarray}
\dd\left(
\raisebox{-3.7ex}{\includegraphics[width=1.4cm]{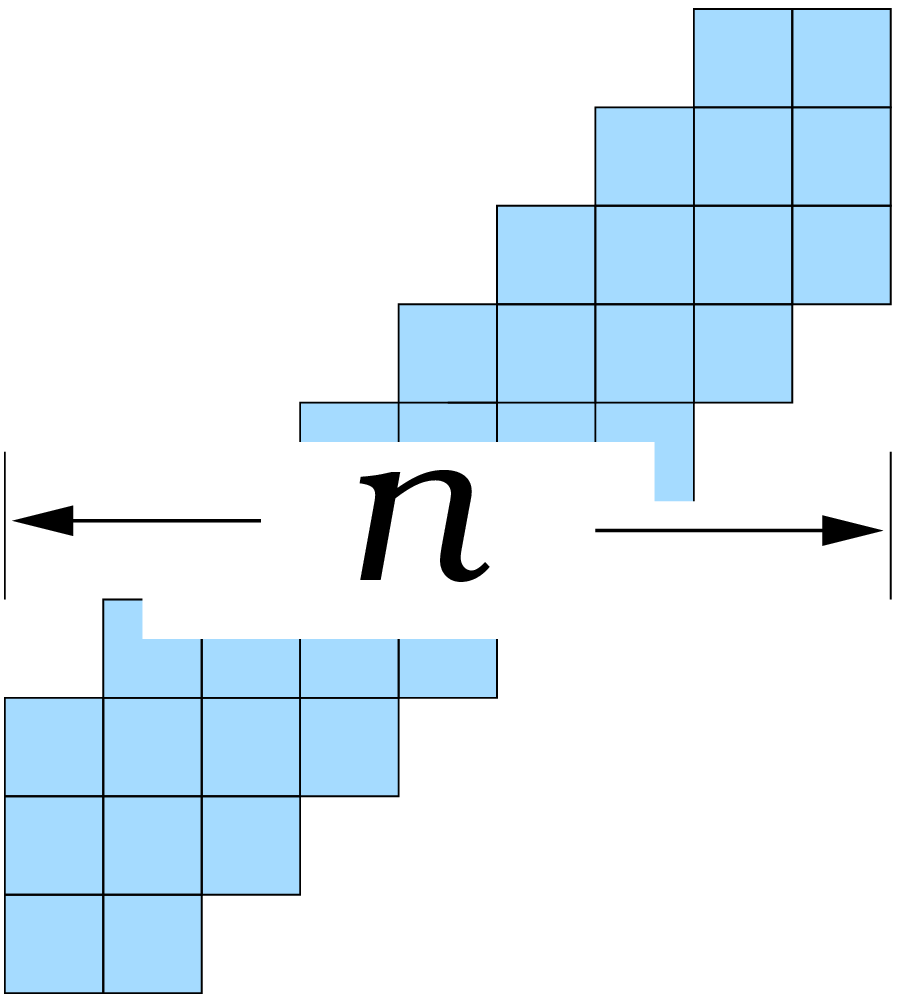}}
\right) 
 &\!\!\!=& 
\!\!(4n-2)!\left( \frac{T_n^2}{((2n-1)!)^2} - \frac{-E_{2n-2}E_{2n}}{(2n-2)!(2n)!}
\right), \label{eq:form4strip1}\\
\dd\left(
\raisebox{-3.9ex}{\includegraphics[width=1.4cm]{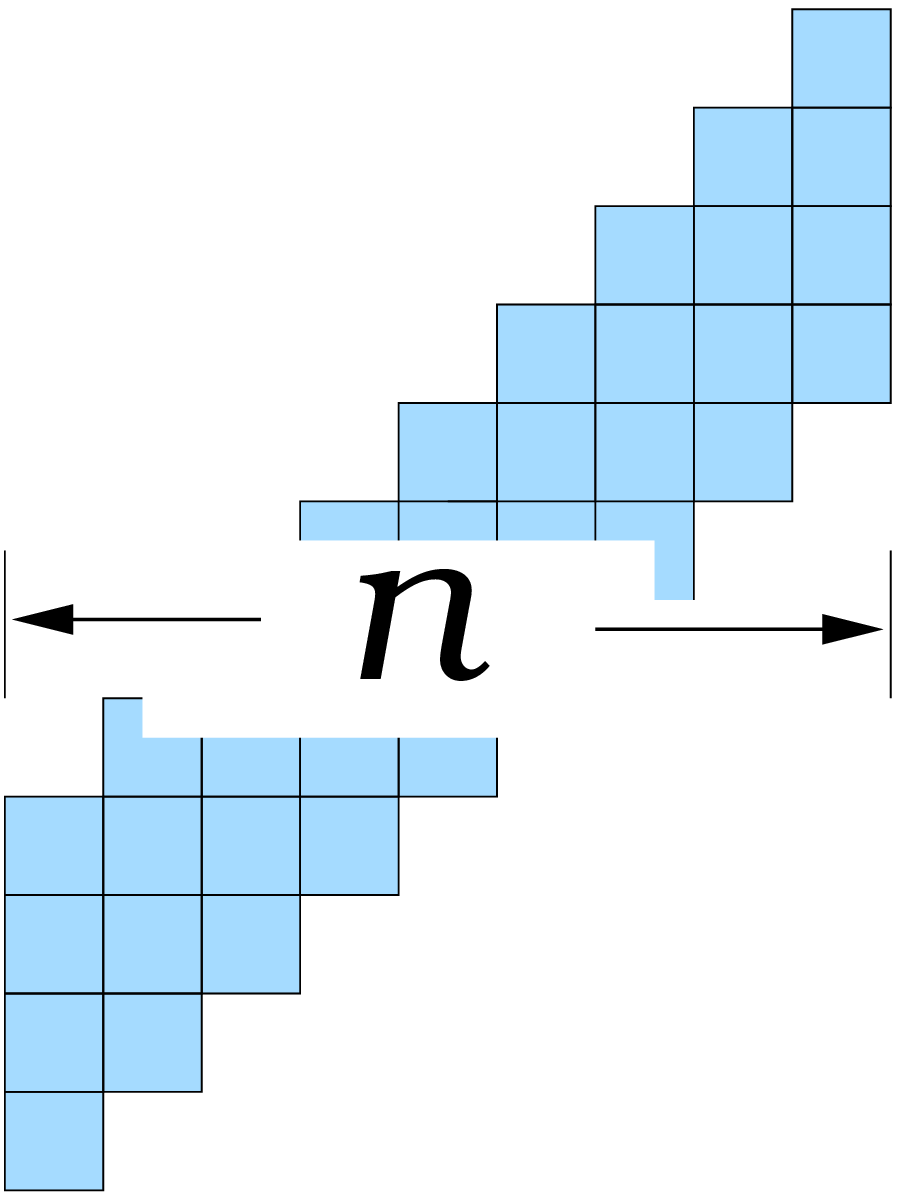}}
\right) 
&\!\!\!\!=& \!\!(4n)!\left(\frac{E_{2n}^2}{((2n)!)^2}-\frac{E_{2n-2}E_{2n+2}}{(2n-2)!(2n+2)!}\right).
\label{eq:form4strip2}
\end{eqnarray}
\end{thm}

\begin{thm}[$5$-strip tableaux] \label{thm5strip}
\begin{eqnarray*}
  \dd\left(
\raisebox{-4.0ex}{\includegraphics[width=1.6cm]{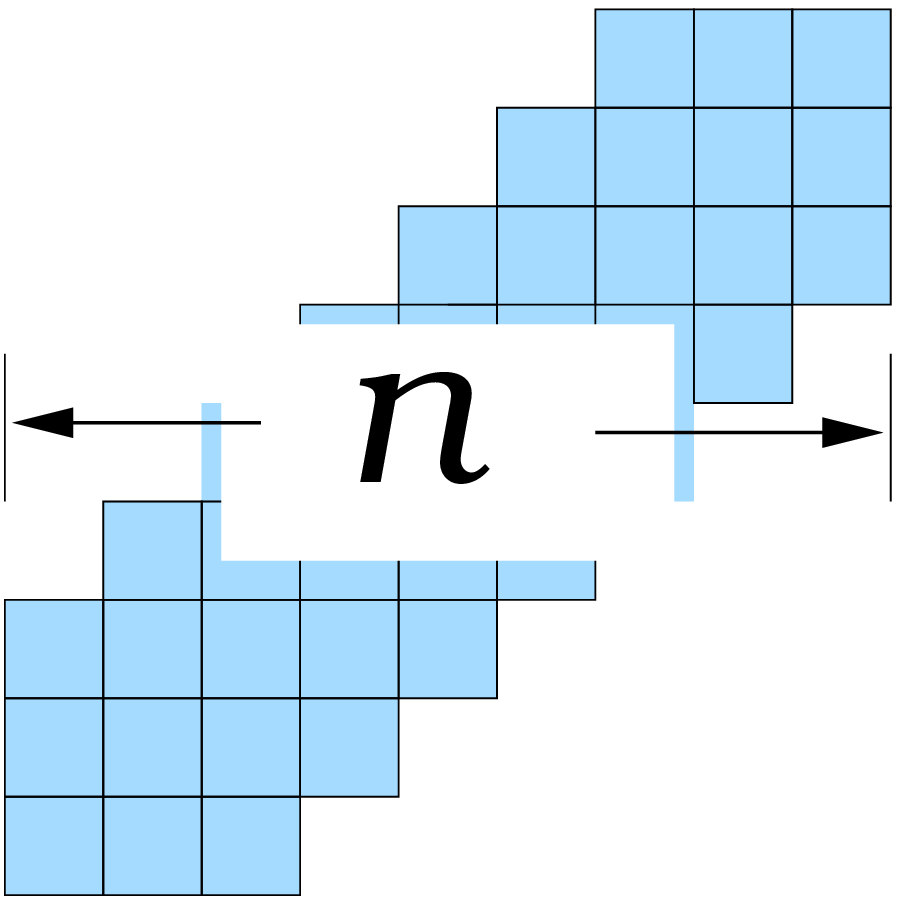}}
\right) &=& 
  \frac{(5n-6)! T_{n-1}^2}{((2n-3)!)^2 2^{4n-6}(2^{2n-2}-1)}. 
\end{eqnarray*}
\end{thm}

The formulas in Theorems \ref{thm3strip}, \ref{thm4strip},
\ref{thm5strip} are special cases of an infinite family of
formulas. To formulate them, we introduce some notation. Let $\bar{A}_n = A_n/n!$ (the volume of the $n$-th up-down polytope, see Section \ref{sectionelkies}). Let $$\tilde{A}_n = \frac{\bar{A}_n}{2^{n+1}-1}, \qquad
\hat{A}_n = \frac{(2^n-1)\bar{A}_n}{2^n (2^{n+1}-1)}.$$
For nonnegative integers $p,q,N$, denote
\begin{eqnarray*}
X_N(p,q) &=&  \sum_{i=0}^{\lfloor p/2\rfloor}
\sum_{j=0}^{\lfloor q/2\rfloor} \frac{(-1)^{i+j}\bar{A}_{N+2i+2j+1}}{(p-2i)!(q-2j)!} \\
& & + (-1)^{\frac{p+1}{2}}\sum_{j=0}^{\lfloor q/2\rfloor} \frac{(-1)^j \bar{A}_{N+p+2j+1}}{(q-2j)!}
1_{[p\text{ odd}]}   \\  & & +
(-1)^{\frac{q+1}{2}}\sum_{i=0}^{\lfloor p/2\rfloor} \frac{(-1)^i \bar{A}_{N+q+2i+1}}{(p-2i)!}
1_{[q\text{ odd}]}
\\ & & + (-1)^{\frac{p+q}{2}+1} \bar{A}_{N+p+q+1} 1_{[p,q\text{ odd}]}, \\ \\
Y_N(p,q) &=&  \sum_{i=0}^{\lfloor p/2\rfloor}
\sum_{j=0}^{\lfloor q/2\rfloor} \frac{(-1)^{i+j}\hat{A}_{N+2i+2j+1}}{(p-2i)!(q-2j)!} \\
& & + (-1)^{\frac{p}{2}}\sum_{j=0}^{\lfloor q/2\rfloor} \frac{(-1)^j \tilde{A}_{N+p+2j+1}}{(q-2j)!}
1_{[p\text{ even}]}   \\  & & +
(-1)^{\frac{q}{2}}\sum_{i=0}^{\lfloor p/2\rfloor} \frac{(-1)^i \tilde{A}_{N+q+2i+1}}{(p-2i)!}
1_{[q\text{ even}]}
\\ & & + (-1)^{\frac{p+q}{2}} \hat{A}_{N+p+q+1} 1_{[p,q\text{ even}]},
\end{eqnarray*}
(here $1_{[x]}$ is 1 if $x$ is true, 0 otherwise). Then we have

\begin{thm} \label{generalthm}
Let $D$ be an $m$-strip diagram as in Figure \ref{figmstrip}. Denote the head Young diagram by $(\lambda_1,\lambda_2,\ldots,\lambda_k)$ and the tail Young diagram by
$(\mu_1,\mu_2,\ldots,\mu_k)$ (where $k=\lfloor m/2\rfloor$). Define the associated numbers $L_i = \lambda_i+k-i$ and $M_i= \mu_i+k-i$, $i=1,\ldots,k$.
If the diagram has a total of $n$ columns, then the number of SYT's of shape $D$ is given by
\begin{equation}\label{eq:ifeven}
\dd(D) = (-1)^{\binom{k}{2}}
|D|!\ 
det\bigg( X_{2n-m+1}(L_i,M_j) \bigg)_{i,j=1,\ldots,k}
\end{equation}
if $m$ is even, or by
\begin{equation}\label{eq:ifodd}
\dd(D) = (-1)^{\binom{k}{2}} |D|!\ 
det\bigg( Y_{2n-m+1}(L_i,M_j) \bigg)_{i,j=1,\ldots,k}
\end{equation}
if $m$ is odd.
\end{thm}

Note that $X_N(p,q)$ and $Y_N(p,q)$ are linear combinations of the zig-zag numbers $A_{N+1},A_{N+2},\ldots,A_{N+p+q+1}$. Thus, the theorem represents $\dd(D)$ for an $m$-strip diagram as a polynomial in the numbers $A_n$, whose complexity depends on the thickness $m$ of the strip but not on the number of columns $n$. For example, for the diagram $D$ in equation \eqref{eq:form4strip2} we have $(\lambda_1,\lambda_2)=(\mu_1,\mu_2)=(1,0)$, so we get
\begin{eqnarray*}\dd(D) &=&  -(4n)! \det\left( \begin{array}{ll} 
X_{2n-3}(0,0) & X_{2n-3}(0,2) \\ X_{2n-3}(2,0) & X_{2n-3}(2,2) \end{array}\right) \\ &=&
 -(4n)! \det\left( \begin{array}{ll} 
 \bar{A}_{2n-2} & \frac12 \bar{A}_{2n-2}-\bar{A}_{2n}   \\ 
\frac12 \bar{A}_{2n-2}-\bar{A}_{2n} & \frac14 \bar{A}_{2n-2}- \bar{A}_{2n} + \bar{A}_{2n+2}
\end{array}\right),
\end{eqnarray*}
which simplifies to give \eqref{eq:form4strip2}. By comparison, trying to use \eqref{eq:aitkeneq} to compute $\dd(D)$ would result in a daunting-looking determinant  of order $n+3$, whose relation to the analogous determinants for the $A_n$'s is unclear.

The following theorem  gives a direct combinatorial meaning to $X_N(p,q)$. The first part is a simple corollary to Theorem \ref{generalthm}, and the second part does not follow from Theorem \ref{generalthm} but can be proved using the same techniques -- see Section \ref{sectionproofgeneral}.

\begin{thm} \label{combmeaning}
For a permutation $\sigma\in S_n$ and $1\le i\le n-1$, we say that $\sigma$ has a descent at $i$ if $\sigma(i)>\sigma(i+1)$. For integers $n,p,q\ge 0$, denote by $\alpha_n$ the number of permutations in 
$S_{2n+p+q}$ whose set of descents is equal to
\begin{multline*} \{ 1,2,3,\ldots,p, p+1,p+3,p+5,\ldots,p+2n-3,p+2n-1, \\ 
p+2n,p+2n+1\ldots,p+2n+q-1 \}.
\end{multline*}
Denote by $\beta_n$ the number of permutations in $S_{2n+p+q}$ whose descent set is equal to
$$\{ 1,2,3,\ldots,p, p+1,p+3,p+5,\ldots,p+2n-3,p+2n-1\}
$$
(equivalently, $\alpha_n$ and $\beta_n$ count SYT's of the shapes shown in Figure \ref{figbelow} below).
Then we have
\begin{eqnarray}
\alpha_n &=& (2n+p+q)! X_{2n-1}(p,q), \label{eq:alphan} \\
\beta_n &=& (2n+1+p+q)! X_{2n}(p,q). \label{eq:betan}
\end{eqnarray}
\end{thm}

Similarly, one can also give a combinatorial meaning to $Y_N(p,q)$ in terms of the number of SYT's for certain $3$-strip diagrams. We leave the precise formulation of this statement to the reader, as an exercise in implementing the techniques of Section \ref{sectionproofgeneral}.

\begin{figure}[h!]
\begin{center}
\begin{tabular}{cc}
{\includegraphics[width=.3\textwidth]{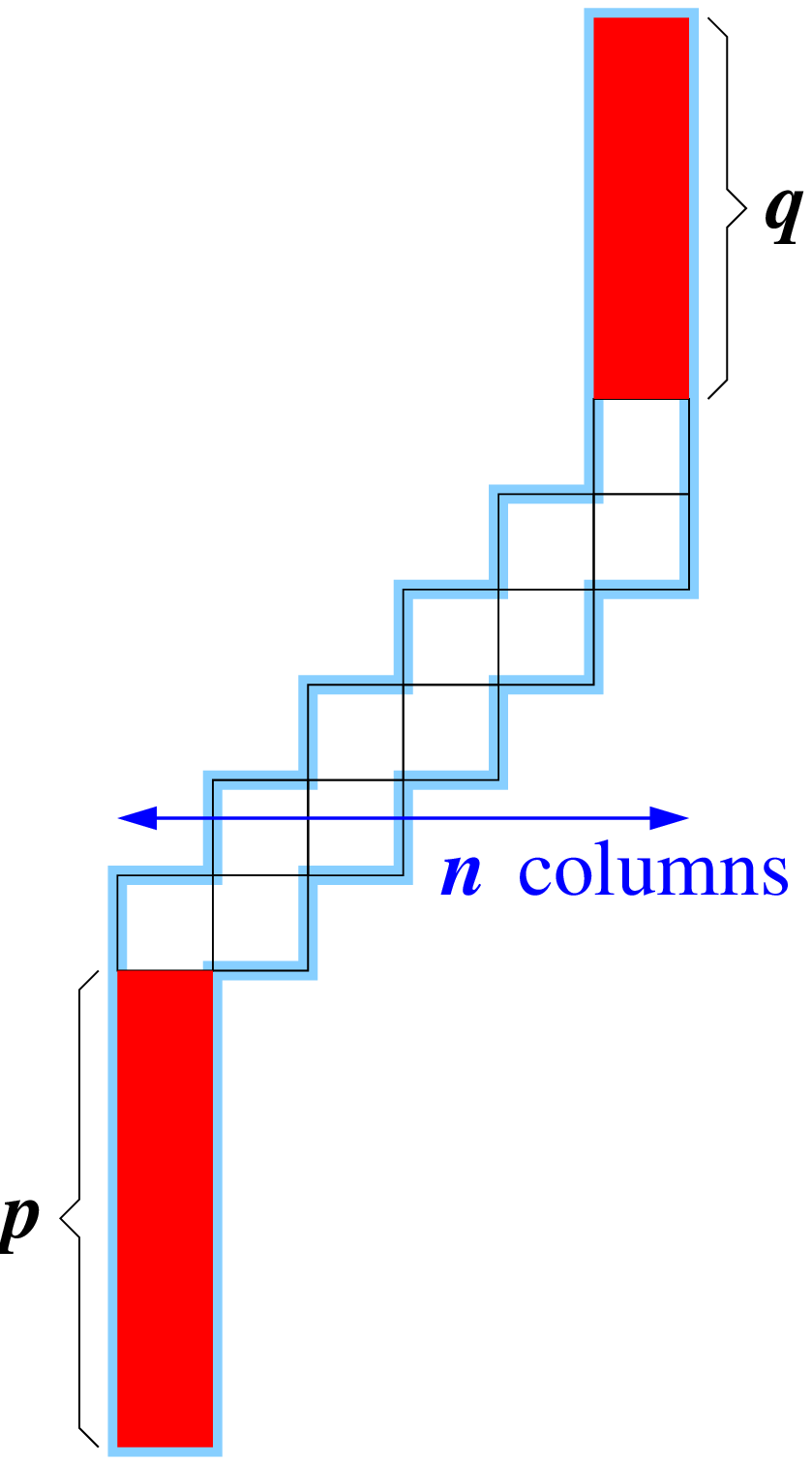}} &
{\includegraphics[width=.35\textwidth]{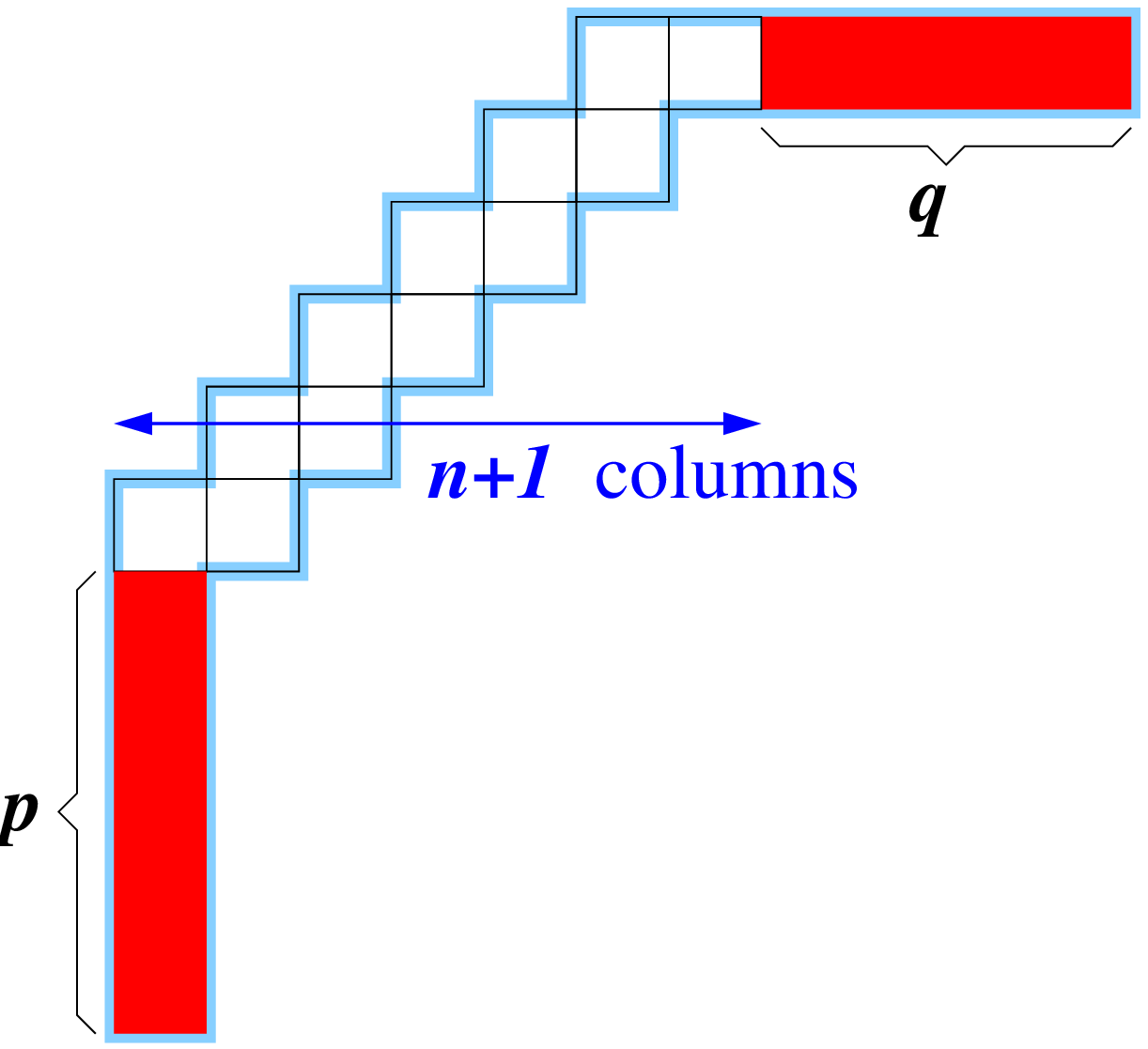}} 
\end{tabular}

\caption{The shapes whose standard Young tableaux are enumerated by $\alpha_n$ and $\beta_n$, respectively. \label{figbelow}}
\end{center}
\end{figure}

In the next sections we develop the tools that will be used to prove
the theorems above. The idea is to use \emph{transfer operators}. This
generalizes the transfer operator approach to up-down permutations
introduced by Elkies \cite{elkies}. Another tool is a geometric-combinatorial identity that expresses the volume of a certain polytope in terms of a Schur symmetric function, and can be thought of as a continuous analogue of a known result on the enumeration of certain shifted Young tableaux.
This identity, Proposition \ref{propschur}, is the subject of Section \ref{sectionschur}.

The understanding of $m$-strip
tableaux that we gain using our analysis of the transfer
operators gives more, and is perhaps more important, than just the proof of the formulas given
above. In fact, more general formulas could be derived, and even more
results such as an understanding of random $m$-strip tableaux of given
shape. Thus, we have the following theorem.

\begin{thm}
The $m$-strip tableaux model is an \emph{exactly solvable} model in the statistical mechanics sense (see Baxter \cite{baxterbook}). That is, the transfer operators can be explicitly diagonalized.
\end{thm}

For more details, see Theorems \ref{thmdiag2k}, \ref{thmdiag2kp1} and comment 1 in Section~\ref{sectioncomments}.

\medskip
To conclude this introduction, we note that the $m$-strip tableaux model is related to the \emph{bead model} studied by Boutillier \cite{boutillier}: The Gibbs measures he constructs seem to describe the limiting distribution of uniform random elements of the order polytope (see Section 4 for the definition) associated with an $m$-strip diagram. The $m$-strip tableaux model is also related to the \emph{square ice} model from statistical physics (see \cite[Chapter 8]{baxterbook}), and our explicit diagonalization of the transfer operators can be thought of as a degenerate case of the so-called  \emph{Bethe ansatz} used in the solution of that model.

\subsection{Acknowledgements}
Theorem \ref{thm3strip}, from which this project evolved, was first discovered and proved when the second-named author was visiting Microsoft Research in 2004, in collaboration with Henry Cohn, Assaf Naor and Yuval Peres. We are also grateful to George Andrews, Omer Angel, Cedric Boutillier, Ehud Friedgut, Ander Holroyd, Rick Kenyon, Christian Krattenthaler, Andrei Okounkov, Igor Pak, Andrea Sportiello, B\'alint Vir\'ag and Doron Zeilberger for helpful discussions.

\section{The transfer operator method of Elkies \label{sectionelkies}}

Noam Elkies \cite{elkies} proposed the following approach to proving \eqref{eq:andreformula}. Define the
$n$-th \emph{up-down polytope} by
$$ P_n = \{ (x_1,\ldots,x_n) \in [0,1]^n\ :\ x_1\le x_2\ge x_3\le x_4\ge \ldots \ \}. $$
Computing $\vol(P_n)$, the volume of $P_n$, in two ways, one first observes that this is simply related to the number $A_n$ of up-down permutations by
$$ \vol(P_n) = \frac{A_n}{n!}. $$
This is because (in probabilistic language) $n$ independently drawn uniform random variables in $[0,1]$ will be in $P_n$ with probability $A_n/n!$, since for each up-down permutation $\sigma$ the probability is $1/n!$ that the random variables will have order structure $\sigma$, and these events have measure 0 pairwise intersections.

On the other hand, by computing the volume as an iterated integral, namely
$$ \vol(P_n) = \int_0^1 dx_1 \int_{x_1}^1 dx_2 \int_0^{x_2} dx_3 \int_{x_3}^1 dx_4 \ldots, $$
upon some more simple manipulation one obtains the following convenient representation in terms of linear operators on the function space $L_2[0,1]$:
\begin{equation}\label{eq:transfer1}
\vol(P_n) = \Big\langle T^{n-1}(\one),\one \Big\rangle, \end{equation}
where $\one$ is the constant function $1$, $T:L_2[0,1]\to L_2[0,1]$ is the self-adjoint operator given by\begin{equation}\label{eq:elkiesop} (Tf)(x) = \int_0^{1-x} f(y)dy,
\end{equation}
and $\langle\cdot,\cdot\rangle$ is the usual scalar product on
$L_2[0,1]$. The operator $T$ is called the \emph{transfer operator}
\footnote{This operator, or rather its discrete version, was applied to
  up-down permutations by V. Arnold, see \cite{Arnold91} and \cite{MSY96}.}. Transfer operators are commonly used in
combinatorics and statistical mechanics (where they are sometimes
called transfer matrices) and in the theory of dynamical systems. In
this case, having obtained the representation \eqref{eq:transfer1},
all that is left to do is to find the orthonormal basis
$(\phi_k)_{k=1}^\infty$ of eigenfunctions, with respective eigenvalues
$\lambda_k$, of $T$ 
(note that $T$ is a self-adjoint operator), since then one gets that
$$ A_n = n! \vol(P_n) = n! \sum_{k=1}^\infty \lambda_k^{n-1} \lang {\bf 1},\phi_k\rang^2. $$
The eigenfunction problem for $T$ leads easily to the solutions
$$ \phi_k(x) = \sqrt{2}\cos\left(\frac{(2k-1)\pi x}{2}\right), \quad \lambda_k = \frac{(-1)^{k-1}\cdot2}{(2k-1)\pi}, \quad k=1,2,3,\ldots $$
whereby one obtains, after some computations that we omit,
\begin{equation}\label{eq:equivalent}
 A_n = \frac{2^{n+2}n!}{\pi^{n+1}}\sum_{k=1}^\infty \frac{(-1)^{(k-1)(n-1)}}{(2k-1)^{n+1}},
\end{equation}
a formula that is equivalent to \eqref{eq:andreformula}.

\section{Markovian polytopes}

We wish to generalize the transfer operator method to the setting of $m$-strip tableaux. 
For skew Young diagrams we can define their associated order polytope and try to apply the same idea to its volume computation. However, there is a certain property of the polytopes associated with up-down permutations that makes usage of the transfer operator method possible (after all, computing volumes of polytopes exactly is in general quite difficult). Borrowing from the language of probability theory, we call this property the Markov property.

\begin{defi}
Let $P\subset \R^d$ be a polytope. If there exists a partition of the coordinate set
$$ \{1,2,\ldots,d\} = C_1 \cup C_2 \cup \ldots \cup C_k \quad \text{(disjoint union)}, $$
such that for all $j=2,\ldots,k-1$, any \emph{$C_j$-section of $P$}, namely a set of the form 
$$
P_{\tilde{x}} := P\cap \{x\in\R^d\ :\ x_{|C_j} = \tilde{x}\},
$$
where $\tilde{x}\in \R^{C_j}$ is fixed, decomposes as a cartesian product of the form
\begin{equation}\label{eq:decompose}
P_{\tilde{x}} = P^{\tilde{x}}_{j-} \times P^{\tilde{x}}_{j+},
\end{equation}
where $P^{\tilde{x}}_{j-} \subset \R^{C_1\times C_2\times \ldots \times C_{j-1}}$ and
$P^{\tilde{x}}_{j-} \subset \R^{C_{j+1}\times C_{j+2}\times \ldots \times C_{k}}$, then we say that $P$ is \emph{Markovian} with respect to the \emph{coordinate filtration} $(C_1,C_2,\ldots,C_k)$.
If $P$ has this property, 
for any $j=1,2,\ldots,k-1$ define the $j$-th \emph{transfer operator}
$ T_j : L_2(\R^{C_j}) \to L_2(\R^{C_{j+1}}) $ associated with $P$ by
$$ \big(T_j(f)\big)(u) = \int_{R^{C_j}} f(v) 1_{\{P_{u,v}\neq \emptyset\}} dv, $$
where $P_{u,v}$ is the section
$$ P_{u,v} = \Big\{ x\in \R^d : x_{|C_j} = v,\ \ x_{|C_{j+1}}=u \Big\}. $$
\end{defi}

The importance of these definitions becomes apparent in the following proposition, which shows how certain volumes can be represented as scalar products in some $L_2$ space.

\begin{prop}\label{propmarkov} In the above notation, if $u \in \R^{C_1}, v\in \R^{C_k}$, then
\begin{equation}\label{eq:markov1}
 \vol_*(P_{u,v}) = \Big\langle T_{k-1}\circ T_{k-2} \circ \ldots \circ T_2 \circ T_1
(\delta_u),\delta_v \Big\rangle_{L_2(\R^{C_k})}.
\end{equation}
Here $\delta_u, \delta_v$ are Dirac delta functions centered around (respectively) $u,v$ in the respective distribution spaces, and $\vol_*$ is the Lebesgue measure of appropriate dimensionality (in this case, $d-|C_j|-|C_{j+1}|$).
Furthermore,
\begin{equation}\label{eq:markov2}
 \vol(P) = \Big\langle T_{k-1}\circ T_{k-2} \circ \ldots \circ T_2 \circ T_1
(\one_{L_2(\R^{C_1})}),\one_{L_2(\R^{C_k})} \Big\rangle_{L_2(\R^{C_k})},
\end{equation}
where $\one$ represents the constant function $1$ in the respective $L_2$ space.
\end{prop}

\begin{proof}
First, note the following trivial identity: If $d_1,d_2,d_3\in\N$ and $A:L_2(\R^{d_1})\to L_2(\R^{d_3}),
B:L_2(\R^{d_2})\to L_2(\R^{d_3})$ are linear operators, then for any functions
$f\in L_2(\R^{d_1}), g\in L_2(\R^{d_2})$ we have
\begin{equation}\label{eq:deltaiden}
\int_{\R^{d_3}} \Big\langle Af,\delta_x\Big\rangle\Big\langle Bg,\delta_x\Big\rangle dx
= \Big\langle Af,Bg\Big\rangle.
\end{equation}

Second, note that from the definition it follows that
the set $P_{j-}^{\tilde{x}}$ in \eqref{eq:decompose}
is, if it's not empty, also Markovian with respect to the filtration 
$(C_1,\ldots,C_{j-1})$, with the same transfer operators $T_1,\ldots,T_{j-2}$ as $P$.

Now, equation \eqref{eq:markov2} clearly follows from \eqref{eq:markov1} by integrating over $u\in \R^{C_1}, v\in \R^{C_k}$. To prove \eqref{eq:markov1}, we use induction together with the above observations.
For each $x_{k-1}\in \R^{C_{k-1}}$ let $P_{u,x_{k-1},v}$ be the section
$$ P_{u,x_{k-1},v} = \Big\{ x\in \R^d : 
x_{|C_j} = v,\ \ x_{|C_{k-1}}=x_{k-1},\ \ x_{|C_{j+1}}=u \Big\}. $$
Then

{\vbox{
\begin{eqnarray*} \vol(P_{u,v}) &=& \ \ \ \int_{\R^{C_{k-1}}} \vol(P_{u,x_{k-1},v})dx_{k-1} 
\qquad\qquad\qquad\qquad\qquad\qquad\end{eqnarray*}

\vspace{-15.0pt}
\begin{eqnarray*}
\ \ \ & \stackrel{\text{(Markov)}}{=} &
\int_{\R^{C_{k-1}}}\vol_*(P_{(k-1)-,u}^{x_{k-1}}) 1_{\{P_{x_{k-1},v}\neq \emptyset\}} dx_{k-1} \\
&\stackrel{\text{(induction)}}{=}&
\int_{\R^{C_{k-1}}}\Big\langle T_{k-2}\ldots T_2 T_1 \delta_u,\delta_{x_{k-1}}\Big\rangle
\Big\langle T_{k-1}\delta_{x_{k-1}},\delta_v\Big\rangle dx_{k-1}
\\ &=&\int_{\R^{C_{k-1}}}\Big\langle T_{k-2}\ldots T_2 T_1 \delta_u,\delta_{x_{k-1}}\Big\rangle
\Big\langle \delta_{x_{k-1}},T_{k-1}^*\delta_v\Big\rangle dx_{k-1} \\ &
\stackrel{\text{(by eq. }\eqref{eq:deltaiden}\text{)}}{=}&
\Big\langle T_{k-2}\ldots T_2 T_1 \delta_u,T_{k-1}^*\delta_v\Big\rangle
\\ &=&
\Big\langle T_{k-1}T_{k-2}\ldots T_2 T_1 \delta_u,\delta_v\Big\rangle,
\end{eqnarray*}
}}

\noindent as claimed.
\end{proof}

\section{The 3-strip \label{sec3strip}}

As a first novel application of the transfer operator technique, we analyze the case of 3-strip tableaux. As in the case of up-down permutations, the first step is to change the discrete problem of enumeration of tableaux to a continuous problem of the computation of a volume of a polytope. Given a skew Young diagram $D$ (considered as a subset of $\N^2$), define its associated \emph{order polytope} as
$$ P_D = \Big\{ x\in [0,1]^D : x_{(i,j)} \le x_{(i',j')}\text{ if }i\le i', j\le j' \Big\}. $$
For the same reasons as before, we have the connection between the discrete and continuous problems: 
\begin{equation} \label{eq:enumerate}
 \vol(P_D) = \frac{\dd(D)}{|D|!}.
\end{equation}

A key step in simplifying the analysis is choosing the correct coordinate filtration for the polytope. There is no unique way of doing this, but a judicious choice will result in a more easily diagonalizable transfer operator. For the case of a 3-strip, we choose the filtration described in Figure \ref{figfiltration3strip} below.

\begin{figure}[h!]
\begin{center}
\resizebox{5cm}{!}{\includegraphics{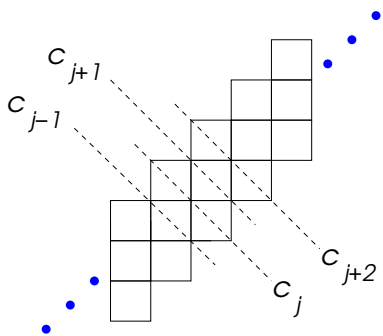}} 
\caption{The coordinate filtration for the 3-strip. \label{figfiltration3strip}}
\end{center}
\end{figure}

Note that in the figure only the part of the filtration corresponding to the body of the diagram is shown. However, for the head and the tail the principle of cutting the diagram along diagonal lines remains the same, and the filtration is constructed accordingly.

Thus, from the figure it is clear that for the body of the diagram we get two transfer operators repeating periodically in alternation. Denote $\Omega=\{(u,v)\in[0,1]^2 : u<v\}$, then the two operators are
$A:L_2[0,1]\to L_2(\Omega), B:L_2(\Omega)\to L_2[0,1]$ given by
$$
(Af)(u,v) = \int_u^v f(x)dx, \quad (Bg)(x) = \int_0^x \int_x^1 g(u,v) dv\,du.
$$
Because the composition $C:=B\circ A$ will repeat periodically, to prove Theorem \ref{thm3strip} we will need to diagonalize this operator. We compute:
\begin{eqnarray*}
(Cf)(x) &=& (B(Af))(x) = \int_0^x \int_x^1 (Af)(u,v) dv\,du
\\&=& \int_0^x \int_x^1 \int_u^v f(y)dy\,dv\,du
\\&=& \int_0^1 f(y)\left( \int_0^{x\wedge y} du \int_{x\vee y}^1 dv \right)dy
\\&=& \int_0^1 f(y)(x\wedge y)(1-x\vee y) dy
\\&=& (1-x)\int_0^x f(y)ydy+x\int_x^1 f(y)(1-y)dy.
\end{eqnarray*}
Now, to find the eigenfunctions:
\begin{eqnarray}
\lambda f(x) &=& (Cf)(x) \nonumber \\&=& (1-x)\int_0^x f(y)ydy+x\int_x^1 f(y)(1-y)dy, \label{eq:eig1}\\
\lambda f'(x) &=& -\int_0^1 yf(y)dy + \int_x^1 f(y)dy, \nonumber \\
\lambda f''(x)&=&-f(x). \nonumber
\end{eqnarray}
From \eqref{eq:eig1} we get the boundary conditions $f(0)=f(1)=0$. This gives the solutions
(scaled to have $L_2$-norm $1$):
$$ \phi_k(x)=\sqrt{2}\sin(\pi k x), \quad \lambda_k=\frac{1}{\pi^2 k^2},\quad k=1,2,3,\ldots $$

\begin{proof}[Proof of Theorem \ref{thm3strip}]
We prove \eqref{eq:form3strip3}; the proof of the other formulas is similar and is omitted. If $D_n$ is the skew Young diagram appearing in \eqref{eq:form3strip3}, then its associated order polytope is Markovian with respect to the coordinate filtration in Figure \ref{figfiltration3strip}. One has to be careful at the ends of the diagram; the first and last transfer operators are easily seen to be, respectively,
$$ (T_{\text{first}}f)(x)=(T_{\text{last}}f)(x)=\int_x^1 f(y)dy. $$
This gives, using Proposition \ref{propmarkov} and \eqref{eq:enumerate}, that the left-hand side of
\eqref{eq:form3strip3} is given by
\begin{eqnarray*} 
\dd(D_n) &=& (3n)!\vol(P_{D_n}) = (3n)!\Big\langle T_{\text{last}} (BA)^{n-1} T_{\text{first}} {\bf 1},{\bf 1} \Big\rangle
\\ &=& (3n)!\Big\langle C^{n-1} (1-x), x \Big\rangle 
\\ &=& (3n)!\sum_{k=1}^\infty \lambda_k^{n-1} \Big\langle x,\phi_k\Big\rangle
\Big\langle 1-x,\phi_k\Big\rangle
\end{eqnarray*}
A quick computation gives that
$ \Big\langle x,\phi_k\Big\rangle=\sqrt{2}(-1)^{k-1}/\pi k$,
$ \Big\langle 1-x,\phi_k\Big\rangle= \sqrt{2}/\pi k$, so
$$ \dd(D_n) = (3n)!\sum_{k=1}^\infty \frac{(-1)^{k-1}2}{\pi^2 k^2(\pi k)^{2(n-1)}}
=\frac{2(3n)!}{\pi^{2n}}\left(1-\frac{2}{2^{2n}}\right)\zeta(2n) $$
(where as usual $\zeta(x)=\sum_{n=1}^\infty n^{-x}$).
Now substituting the classical identity
\begin{equation}\label{eq:classical}
\zeta(2n) = \sum_{k=1}^\infty \frac{1}{k^{2n}} = \frac{(-1)^{n-1}\pi^{2n} 2^{2n-1} B_{2n}}{(2n)!}
\end{equation}
(together with \eqref{eq:viarelation}) gives \eqref{eq:form3strip3}.
\end{proof}

\section{The 4-strip \label{sec4strip}}

While the analysis so far has been relatively straightforward, the case of the $4$-strip is the first case where one encounters relative difficulty in diagonalizing the transfer operator, which now works on a 2-dimensional domain. As before, we choose the coordinate filtration obtained by cutting the diagram along diagonal lines, as shown in the figure below.

\begin{figure}[h!]
\begin{center}
\resizebox{5cm}{!}{\includegraphics{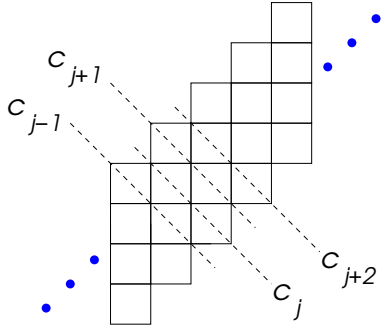}} 
\caption{The coordinate filtration for the 4-strip. \label{figfiltration4strip}}
\end{center}
\end{figure}

Again, we get the following two transfer operators repeating periodically in alternation. Denote as before
$\Omega=\{(u,v)\in[0,1]^2 : u<v\}$, then we have that $A,B:L_2(\Omega)\to L_2(\Omega)$ are given by
\begin{eqnarray*}
(Af)(x,y) &=& \int_0^x \int_x^y f(u,v) dv\,du, \\
(Bg)(u,v) &=& \int_u^v \int_v^1 g(x,y) dx\,dy.
\end{eqnarray*}
A small simplification is to note that there is a certain symmetry, in that $A$ and $B$ are conjugate to each other: $B=CAC$, where $C:L_2(\Omega)\to L_2(\Omega)$ is the reflection operator
$$ (Cf)(x,y) = f(1-y,1-x). $$
So, instead of diagonalizing the composition $BA=CACA$ we can diagonalize its simpler ``square root'' $CA$, given by
\begin{equation} (CAf)(x,y) = \int_0^{1-y} \int_{1-y}^{1-x} f(u,v) dv\,du. \label{eq:operator} \end{equation}
So we are looking for $\lambda, f$ which are solutions of
$$ \lambda f(x,y) = \int_0^{1-y} \int_{1-y}^{1-x} f(u,v) dv\,du. $$
Differentiating once with respect to each variable gives
\begin{equation} \lambda \frac{\partial^2 f}{\partial x \partial y} = f(1-y,1-x). \label{eq:diffonce}
\end{equation}
Differentiating again w.r.t. $x,y$ and substituting in \eqref{eq:diffonce} gives the PDE
\begin{equation} \frac{\partial^4 f}{\partial^2 x \partial^2 y} = \frac{1}{\lambda^2}f(x,y).
\label{eq:pde4}
\end{equation}
The boundary conditions are easily seen to be
\begin{equation}\label{eq:boundarycond4}
 f(x,1) \equiv 0, \quad f(x,x) \equiv 0, \quad f_x(0,y) \equiv 0,
\end{equation}
a mixture of Dirichlet- and Neumann-type conditions on the three boundary sides of $\Omega$.

Now, it may be verified that the $L_2$-normalized solutions to this boundary value problem are given by
\begin{eqnarray} \phi_{j,k}(x,y) &=& 2 \det\left( \begin{array}{ll} 
\cos\left(\frac{\pi(2j-1)x}{2}\right) & \cos\left(\frac{\pi(2j-1)y}{2}\right) \nonumber\\
\cos\left(\frac{\pi(2k-1)x}{2}\right) & \cos\left(\frac{\pi(2k-1)y}{2}\right) \end{array}
\right),\\ \lambda_{j,k}&=&\frac{4(-1)^{j+k-1}}{\pi^2(2j-1)(2k-1)}, \qquad k>j>0\text{ integers},
\label{eq:eigensystem}
\end{eqnarray}
(note that the eigenfunctions are parametrized by two integers).
While verification is easy, two related questions that should be addressed are: how these solutions can be derived rather than guessed; and how to prove that these solutions span all eigenspaces.
We answer the first question and then point out how this essentially contains the answer to the second. The PDE \eqref{eq:pde4} is of a very simple form that suggests trying to use Fourier series and the method of separation of variables. However, the domain $\Omega$ is not of a shape suitable for the application of this method. But looking at the boundary conditions \eqref{eq:boundarycond4} suggests looking for a function on a larger domain that satisfies the symmetry and anti-symmetry conditions
\begin{eqnarray}
f(y,x) &=& -f(x,y), \nonumber \\
f(-x,y) &=& f(x,-y) = f(x,y), \label{eq:symmetry4} \\
f(x,2-y) &=& f(2-x,y) = -f(x,y) \nonumber
\end{eqnarray}
since such a function will automatically satisfy \eqref{eq:boundarycond4}. 
Now, any function satisfying \eqref{eq:symmetry4} is periodic in both variables with period 4, i.e., satisfies $f(x+4,y)=f(x,y+4)=f(x,y)$, and its values everywhere are determined by its values on $\Omega$.
Thus one gets a problem suitable for the application of the method of separation of variables: Expand $f$ in a Fourier series
$$ f(x,y) = \sum_{a,b\in\Z} c_{a,b} \exp\left(\frac{\pi i}{2}(ax+by)\right). $$
The symmetry conditions \eqref{eq:symmetry4} translate to the following equations on the coefficients $c_{a,b}$:
$$ c_{-a,b}=c_{a,b}, \quad c_{a,-b}=c_{a,b}, \quad c_{b,a}=-c_{a,b}, $$
$$ c_{a,b} = (-1)^{b+1}c_{a,b}, \quad c_{a,b}=(-1)^{a+1}c_{a,b}. $$
In particular the last conditions imply that $a,b$ are odd if $c_{a,b}\neq 0$.
Eq. \eqref{eq:diffonce} implies the additional equation
$$ -\frac{\pi^2}{4}\lambda a b\, c_{a,b} = -i^{a+b} c_{a,b}, $$
Denoting $\lambda = 4i^{a+b}/\pi^2 a b$ for some (and therefore all) $a,b$ for which $c_{a,b}\neq 0$, it is easy to verify from these relations that $f$ is a linear combination of the functions $\phi_{j,k}$ with $k>j>0$ and  $ |a b| = (2j-1)(2k-1)$ (note that some eigenvalues do have multiplicity greater than 1, corresponding to different representations of $|ab|$ as products of distinct odd positive integers).

It remains to answer the second question posed above, of showing that the system \eqref{eq:eigensystem} is indeed a complete eigensystem for the operator $CA$ . The answer is encoded in the previous discussion, but one has to argue more formally now. Let $X$ be the subset of $L_2([-2,2]^2)$ of functions satisfying 
the symmetry conditions \eqref{eq:symmetry4}, and let $Y$ be the subset of
$L_2(\Omega)$ of functions satisfying the boundary conditions 
\eqref{eq:boundarycond4}. Let $D$ be the linear operator on $X$ 
given by
\begin{equation} (Df)(x,y) = \int_0^{1-y} \int_{0}^{1-x} f(u,v)
  dv\,du, \label{eq:op_anti}  
\end{equation}
manifestly a compact self-adjoint operator.

Note that 
$Df$ can be also defined by the right-hand side of \eqref{eq:operator}
due to the symmetry conditions imposed on elements of $X$. 
Let $P:X\to Y$ be defined by $Pf = f_{|\Omega}$. Then $P$ is clearly
an isomorphism of Hilbert spaces, and $D$ is
conjugate to $CA$ under $P$, i.e., $CA = PDP^{-1}$. Therefore $f\in Y$
is an eigenfunction of $CA$ iff $P^{-1}f$ is an eigenfunction of
$D$. As the eigenfunctions of $D$ form a complete orthonormal system in $X$,
we have proved: 

\begin{prop} \label{propstrip4diag}
The system \eqref{eq:eigensystem} is a complete normalized eigensystem
for the operator $CA$. 
\end{prop}

\begin{proof}[Proof of Theorem \ref{thm4strip}]
Denote by $F_n$ and $G_n$ the skew Young diagrams appearing on the left-hand sides of \eqref{eq:form4strip1}, \eqref{eq:form4strip2}, respectively. Following the same reasoning as before using Proposition \ref{propmarkov} and \eqref{eq:enumerate}, we have that
\begin{eqnarray*}
\dd(F_n) &=& (4n-2)!\Big\langle (CA)^{2n-3}(y-x),y-x\Big\rangle, \\
\dd(G_n) &=& (4n)!\Big\langle (CA)^{2n-3}\left(\frac{y^2-x^2}{2}\right),\frac{y^2-x^2}{2}\Big\rangle. \\
\end{eqnarray*}
Using these formulas and Proposition \ref{propstrip4diag}, after a computation one finally arrives at \eqref{eq:form4strip1}, \eqref{eq:form4strip2}, using \eqref{eq:classical} and another classical identity due to Euler, namely
$$ \sum_{k=1}^\infty \frac{(-1)^{k-1}}{(2k-1)^{2n+1}} = \frac{(-1)^n \pi^{2n+1} E_{2n}}{2^{2n+2}(2n)!}. $$
We leave verification of these straightforward details to the reader. Also note that these computations are superseded by our more general proof of Theorem \ref{generalthm} in Section \ref{sectionproofgeneral}.
\end{proof}

\section{The $2k$-strip and the $(2k+1)$-strip}

Having analyzed the case of $4$-strip diagrams, the results are now
easily generalized to the $2k$-strip and the
$(2k+1)$. Here are the results for the $2k$-strip.
The coordinate filtration is again constructed by cutting
the diagram along diagonals. This leads to a transfer operator on 
$L_2(\Omega_k)$, where
$$ \Omega_k = \Big\{ x\in [0,1]^k : x_1\le x_2\le \ldots \le x_k \Big\}.$$
The transfer operator is given by
\begin{multline} (S_{2k}f)(x_1,\ldots,x_k) 
= \\
 \int_0^{1-x_k} \int_{1-x_k}^{1-x_{k-1}}
\int_{1-x_{k-1}}^{1-x_{k-2}} \ldots \int_{1-x_2}^{1-x_1}
f(y_1,\ldots,y_k) dy_k\,dy_{k-1}\,\ldots\,dy_1. \label{eq:transfer2k}
\end{multline}
In this representation we have already taken advantage of the symmetry 
using the ``square root'' trick as
in the case of the 4-strip to reduce two conjugate transfer operators repeating
in alternation to a single operator. Diagonalizing $S_{2k}$ leads to the following boundary value problem:
$$\frac{\partial^{k}f(x_1,x_2,\ldots,x_k)}{\partial x_1 \partial x_2 \ldots \partial x_k}
= \frac{(-1)^k}{\lambda} f(1-x_k,\ldots,1-x_2,1-x_1),$$

\vspace{-20.0pt}
\begin{eqnarray*}
f &\equiv& 0 \text{ on }x_1 \equiv x_2,
x_2\equiv x_3, \ldots, x_{k-1}\equiv x_k, x_k\equiv 1, \\
f_{x_1} &\equiv& 0\text{ on }x_1\equiv 0.
\end{eqnarray*}
The ideas of Section \ref{sec4strip} can be used to prove:

\begin{thm}[Diagonalization of the $2k$-strip transfer operator] \label{thmdiag2k}
The system of functions and associated eigenvalues
\begin{eqnarray*}
\phi_{j_1,\ldots,j_k}(x_1,\ldots,x_k) &=& 2^{k/2} \det\left(
\cos\left(\frac{\pi (2j_p-1)x_q}{2}\right) \right)_{1\le p,q\le k}, \\
\lambda_{j_1,\ldots,j_k} &=& \frac{2^k(-1)^{\binom{k}{2}+\sum_p (j_p+1)}}{\pi^k (2j_1-1)(2j_2-1)\ldots(2j_k-1)}, \\ & & \ j_1>j_2>\ldots>j_k>0
\text{ integers,}
\end{eqnarray*}
is a complete orthonormal eigensystem for the operator $S_{2k}$ in 
\eqref{eq:transfer2k}.
\end{thm}

For the $(2k+1)$-strip, we have similar results generalizing the analysis 
of Section \ref{sec3strip}. In this case, we have two operators
$A:L_2(\Omega_{k+1})\to L_2(\Omega_k),
B:L_2(\Omega_k)\to L_2(\Omega_{k+1})$ given by
\begin{multline*}
(A g)(x_1,\ldots,x_k)
 = \\ \int_0^{x_1} \int_{x_1}^{x_2}
\int_{x_2}^{x_3}\ldots \int_{x_{k-1}}^{x_k} \int_{x_k}^1
g(y_1,\ldots,y_{k+1})dy_{k+1}\,\ldots\,dy_1,
\end{multline*}

\vspace{-16.0pt}
\begin{multline*}
\ \ (B h)(y_1,\ldots,y_{k+1})
 \ = \\ \int_{y_1}^{y_2} \int_{y_2}^{y_3}
\int_{y_3}^{y_4}\ldots \int_{y_k}^{y_{k+1}} 
h(x_1,\ldots,x_k)dx_k\,\ldots\,dx_1.
\end{multline*}
The relevant operator to diagonalize is the composition $A B$. This leads
to the following boundary value problem:
$$\frac{\partial^{2k}f}{\partial^2 x_1 \partial^2 x_2 \ldots \partial^2 x_k}
= \frac{1}{\lambda^2} f, \qquad f_{\big| \partial \Omega_k}\equiv 0$$
The solution is given by the following theorem.

\begin{thm}[Diagonalization of the $(2k+1)$-strip transfer operator] \label{thmdiag2kp1}
The system of functions and associated eigenvalues
\begin{eqnarray*}
\phi_{j_1,\ldots,j_k}(x_1,\ldots,x_k) &=& 2^{k/2} \det\Big(
\sin\left(\pi j_p x_q\right) \Big)_{1\le p,q\le k}, \\
\lambda_{j_1,\ldots,j_k} &=& \frac{1}{\pi^{2k} j_1^2 j_2^2\ldots j_k^2}, \\ & &  j_1>j_2>\ldots>j_k>0
\text{ integers,}
\end{eqnarray*}
is a complete orthonormal eigensystem for the operator $A B$.
\end{thm}

\section{A Schur function identity \label{sectionschur}}

In this section, we prove an identity that will be used in the next section in the proof of Theorem \ref{generalthm}.

Let $\lambda=\left(\lambda_1\geq\lambda_2\geq\ldots\geq\lambda_k\ge 0\right)$ be a
Young diagram, possibly with some parts being 0.
Consider the head (or tail) polytope $P_\lambda$,
defined by the system of linear inequalities:
\begin{eqnarray*}
0\leq y_{i,j}\leq 1&& 1\leq i\leq k,\ \  -(k-i)\leq j\leq \lambda_i,\\
y_{i,j}\geq y_{i+1,j}&&\text{ where defined,}\\
y_{i,j}\geq y_{i,j+1}&&\text{ where defined,}
\end{eqnarray*}
and for ${\bf x}=(x_1,\ldots,x_k)$ in the simplex $\Omega_k$, denote by $P_\lambda({\bf x})$ the section
$$ P_\lambda({\bf x}) = \Big\{ y\in P_\lambda : y_{k+1-i,-(i-1)} = x_i,\ \  i=1,2,\ldots,k \Big\} $$
(see Figure \ref{head}).

\begin{figure}[h!]
\begin{center}
{\includegraphics[width=.5\textwidth]{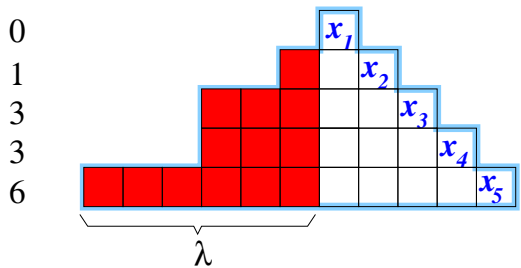}}
\caption{The polytope $P_\lambda$ is the order polytope of the diagram (which is \emph{not} a skew Young diagram) formed by leaning the Young diagram $\lambda$ against a triangular (``staircase'') diagram, and $P_\lambda({\bf x})$ is its section where the values along the main diagonal are $x_1,\ldots,x_k$. In this example $\lambda=(6,3,3,1,0)$. \label{head}}
\end{center}
\end{figure}

Obviously, the volume of $P_\lambda({\bf x})$ is a polynomial function on
the simplex $\Omega_k$. More precisely:
\begin{prop} \label{propschur}
Let $L_i= \lambda_i+k-i,\ \  1\leq i\leq k$. 
The volume of the polytope $P_\lambda({\bf x})$ (of the appropriate dimensionality) is 
\begin{equation}\label{eq:monomial}
\vol_*(P_\lambda({\bf x}))=\frac{1}{\prod_{i=1}^k L_i!} \det\left(
  x_i^{L_j}\right)_{1\leq i,j\leq k} 
\end{equation}
\end{prop}

Proposition \ref{propschur} generalizes a well-known fact corresponding to the case $\lambda=(0,0,\ldots,0)$, see for example \cite[Lemma 1.12]{baryshnikov} and \cite[Section 4]{aztecdiamonds}.
It is known to experts, though perhaps in a slightly different form, and can be proved using the Karlin-McGregor theory of non-intersecting paths. We will give a direct proof by induction.
Note that the expression on the right-hand side of \eqref{eq:monomial} is equal to $\left(\prod_{i=1}^k L_i!\right)^{-1}s_\lambda({\bf x})V({\bf x})$, 
where $s_\lambda$ is the Schur symmetric function associated to $\lambda$ and $V({\bf x})=\prod_{i<j}
(x_i-x_j)$ is the Vandermonde determinant.

\begin{proof}
We will proceed by induction on $|\lambda|+k$ (where $|\lambda|=\lambda_1+\ldots+\lambda_k$). The induction base is trivial. For the inductive step, we divide into two cases according to whether $\lambda_k>0$ or $\lambda_k=0$. In the former case,  the inductive hypothesis holds for the smaller Young diagram
$$
\lambda'=\left(\lambda_1-1,\lambda_2-1,\ldots,\lambda_k-1\right),
$$
and the volume of the polytope $P_{\lambda}({\bf x})$ can be represented as
\begin{equation} \label{eq:schur_induc}
\int_0^{x_1}\int_{x_1}^{x_2}\ldots\int_{x_{k-1}}^{x_k}
\vol_*[P_{\lambda'}(y_1,\ldots,y_k)]dy_k\, dy_{k-1}\ldots dy_1
\end{equation}
(this is essentiallly a transfer operator computation, since the
diagram formed by $\lambda$ with its adjoining triangular part can be
obtained from the corresponding diagram for $\lambda'$ by adding one
diagonal). Using the inductive hypothesis, 
the integrand is antisymmetric, and therefore the expression in \eqref{eq:schur_induc} can be rewritten as
\begin{equation} \label{eq:schur_induc2}
\int_0^{x_1}\int_{0}^{x_2}\ldots\int_{0}^{x_k}
\vol_*[P_{\lambda'}(y_1,\ldots,y_k)]dy_k\, dy_{k-1}\ldots dy_1
\end{equation}
(similar reasoning was used in the proof of Proposition
\ref{propstrip4diag}), which in turn, again by the inductive hypothesis, is equal to
\begin{eqnarray} \nonumber
\int_0^{x_1}\ldots\int_{0}^{x_k} \det\left(\frac{1}{(L_j-1)!} y_i^{L_j-1}\right)_{i,j} dy_k\ldots dy_1
\qquad\qquad
\\ = \ \ \det\left(\frac{1}{L_j!} x_i^{L_j}\right)_{i,j}, 
\qquad\qquad\qquad\qquad\qquad\qquad\qquad  \label{eq:schur_induc3}
\end{eqnarray}
proving the inductive step in this case.

The other induction step deals with the second case where
$\lambda_k=0$. Here, we use the inductive hypothesis for the Young
diagram 
$$\lambda' = (\lambda_1,\ldots,\lambda_{k-1})$$
(the same diagram, but we attach to it a smaller triangle of order
$k-1$ only). The relation between the two polytopes $P_{\lambda}({\bf
  x}), P_{\lambda'}({\bf y})$ can be seen (again using an iterated
integral which is really a transfer operator computation) to be 
\begin{equation*}
\vol_*(P_\lambda({\bf x})) = \!
\int_{x_1}^{x_2}\!\int_{x_1}^{x_3}\!\ldots\!\int_{x_{1}}^{x_k}
\!\!\vol_*(P_{\lambda'}(y_2,\ldots,y_k))dy_{k}\ldots dy_3\,dy_2.
\end{equation*}
This can be dealt with
similarly to (\ref{eq:schur_induc3}), simplifying eventually to the expression 
$$ \frac{1}{\prod_{i=1}^{k-1} L_i!} \det\left( \begin{array}{llll}
x_1^{L_1} & x_2^{L_1} & \ldots & x_k^{L_1} \\
\vdots & \vdots &  & \vdots \\
x_1^{L_{k-1}} & x_2^{L_{k-1}} & \ldots & x_k^{L_{k-1}} \\
1 & 1 & \ldots & 1 
\end{array}\right)
$$
which is exactly the right-hand side of \eqref{eq:monomial} (since
$L_k=0$ in this induction step). 
\end{proof}

It is worth noting that this proof, similarly to the computations in the previous sections, is based on cutting the diagram along diagonal lines and computing recursively using transfer operators.

\section{Proof of Theorems \ref{generalthm} and \ref{combmeaning} \label{sectionproofgeneral}}

\begin{lem}[Andreief's Formula \cite{andreief}] \label{andreieff}
If $(\Omega,\mu)$ is a measure space, and $f_1,f_2,\ldots,f_k, g_1,g_2,\ldots,g_k$ are real-valued functions on $\Omega$, then
\begin{multline*} \qquad
\int_{\Omega^m} \det\bigg( f_i(x_l) \bigg)_{1\le i,l\le k}  
\det\bigg(g_j(x_l) \bigg)_{1\le l,j\le k} 
d\mu(x_1) \ldots
d\mu(x_k) \\ = k!\ \det\bigg( \int_\Omega f_i(x) g_j(x) d\mu(x) \bigg)_{i,j}. \qquad
\end{multline*}
\end{lem}

\begin{proof}
$$\int_{\Omega^m} \det\bigg( f_i(x_j) \bigg)_{i,j} 
\det\bigg(g_j(x_l) \bigg)_{1\le l,j\le k} 
d\mu(x_1) \ldots
d\mu(x_k) \qquad\qquad\qquad\qquad\qquad\qquad$$

\vspace{-20.0pt}
\begin{eqnarray*}
&=&
\int_{\Omega^m} \sum_{\sigma\in S_k} \sum_{\pi\in S_k} \varepsilon(\sigma)\varepsilon(\pi) \prod_{l=1}^k f_{\sigma(l)}(x_l) g_{\pi(l)}(x_l) d\mu^{\otimes k}(x_1,\ldots,x_k)
\\ &=&
\sum_{\sigma\in S_k} \sum_{\pi\in S_k} \varepsilon(\sigma)\varepsilon(\pi) \prod_{i=1}^k
\left( \int_\Omega f_{\sigma(l)}(x)g_{\pi(l)}(x) d\mu(x)\right) \\&=&
\sum_{\sigma\in S_k} \sum_{\pi\in S_k} \varepsilon(\pi \sigma^{-1}) \prod_{j=1}^k
\left( \int_\Omega f_j(x)g_{\pi \sigma^{-1}(j)}(x) d\mu(x) \right)\\&=&
k!\ \det\bigg( \int_\Omega f_i(x) g_j(x) d\mu(x) \bigg)_{i,j}.
\end{eqnarray*}
\end{proof}

\begin{figure}[h!]
\begin{center}
{\includegraphics[width=.5\textwidth]{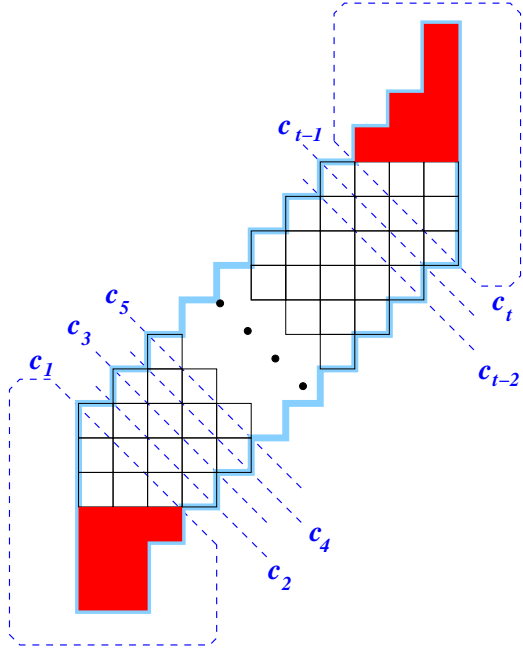}}
\caption{The coordinate filtration for the diagram $D$. \label{filtrationfig}}
\end{center}
\end{figure}

\begin{proof}[Proof of Theorem \ref{generalthm}]

We now compute $\dd(D)$, where $D$ is described in Theorem
\ref{generalthm}, for the case when $m=2k$ is even. The case of odd
$m$ is done similarly and is left to the reader. We choose the usual
filtration $(C_1,C_2,\ldots,C_t)$, except that the first coordinate
set $C_1$ represents the tail and the triangular part of the body to
which it is attached, and similarly the last coordinate set $C_t$
represents the head and the triangular part of the body to which it is
attached, as shown in Figure \ref{filtrationfig}. Proposition \ref{propschur} gives us
a good starting position for this computation, since it implies that
the first and last transfer operators, $T_1$ and $T_t$, satisfy
\begin{eqnarray*} 
(T_1\one)(x_1,\ldots,x_k) &=& \frac{1}{\prod_{i=1}^k L_i!}
\det\bigg( x_i^{L_j} \bigg)_{1\le i,j\le k} =: \psi_\lambda(x_1,\ldots,x_k), \\
(T_t^*\one)(x_1,\ldots,x_k) &=& \frac{1}{\prod_{i=1}^k M_i!}
\det\bigg( x_i^{M_j} \bigg)_{1\le i,j\le k}
=: \psi_\mu(x_1,\ldots,x_k).
\end{eqnarray*}
The transfer operator for the body part is $S_{2k}$ defined in \eqref{eq:transfer2k}.
Therefore we get using Proposition \ref{propmarkov} that
$$
\frac{\dd(D)}{|D|!}= \lang T_t \circ S_{2k}^{2n-m+1} \circ T_1 \one, \one \rang
= \lang S_{2k}^{2n-2m+1} \psi_\lambda, \psi_\mu \rang $$
Expanding this using Theorem \ref{thmdiag2k}, we get, in the notation of that theorem,
\begin{equation} \label{eq:continuing}
\frac{\dd(D)}{|D|!} =
\sum_{j_1>j_2>\ldots>j_k>0}\lambda_{j_1,\ldots,j_k}^{2n-m+1}
\lang \psi_\lambda,\phi_{j_1,\ldots,j_k} \rang\lang \psi_\mu,\phi_{j_1,\ldots,j_k} \rang.
\end{equation}
To compute $\lang \psi_\lambda,\phi_{j_1,\ldots,j_k} \rang, \lang \psi_\mu,\phi_{j_1,\ldots,j_k} \rang$,
we use Lemma \ref{andreieff}: 

{\vbox{
$$
\lang \psi_\lambda,\phi_{j_1,\ldots,j_k} \rang \qquad\qquad\qquad\qquad\qquad\qquad\qquad\qquad
\qquad\qquad\qquad\qquad
$$

\vspace{-20.0pt}
\begin{eqnarray*}
& =& \int\ldots\int_{\Omega_k} 
\det\bigg(\frac{x_r^{L_s}}{L_s!}
\bigg)_{r,s}\!\!\! \det\bigg(\!\! \sqrt{2}\cos\left(\frac{\pi}{2} (2j_r-1)x_s\right)\!\!\bigg)_{r,s}\! dx_1\ldots\! dx_k
\\ &=& \frac{1}{k!}\int\ldots\int_{[0,1]^k}
\det\bigg(\frac{x_r^{L_s}}{L_s!}
\bigg)_{r,s}\!\!\! \det\bigg(\!\! \sqrt{2}\cos\left(\frac{\pi}{2} (2j_r-1)x_s\right)\!\!\bigg)_{r,s}\! d{\bf x}
\\&=& 2^{k/2}\det\left(\frac{1}{L_s!} \int_0^1 x^{L_s} \cos\left(\frac{\pi}{2}(2j_r-1)x\right)dx \right)_{r,s}
\\ &=& 2^{k/2}
\det\bigg( I(L_s,j_r)\bigg)_{r,s},
\end{eqnarray*}
}}
where we have denoted
$$ I(a,j) = \frac{1}{a!}\int_0^1 x^a \cos\left(\frac{\pi}{2}(2j-1)x\right)dx. $$
Now, continuing \eqref{eq:continuing}, we get, \emph{again} using Lemma \ref{andreieff} (this time used with a discrete measure), that
$$ (-1)^{\binom{k}{2}}\cdot 
\frac{\dd(D)}{|D|!} \qquad\qquad\qquad\qquad\qquad\qquad\qquad\qquad\qquad\qquad\qquad\qquad $$

\vspace{-18.0pt}
\begin{eqnarray}
&=& \nonumber
\!2^k \!\!\!\!\!\!
\sum_{j_1>\ldots>j_k>0} \prod_{r=1}^k \left(\frac{2(-1)^{j_r+1}}{\pi(2j_r-1)}\right)^{2n-m+1}
\!\!\!\!\!\!\!\!\!
\det\bigg( I(L_s,j_r)\bigg)_{r,s} \!\!\!\!\!\!\det\bigg( I(M_s,j_r)\bigg)_{r,s}
\\ \nonumber &=& 
\frac{2^k}{k!}
\sum_{{\bf j}\in\mathbb{N}^k} \prod_{r=1}^k \left(\frac{2(-1)^{j_r+1}}{\pi(2j_r-1)}\right)^{2n-m+1}
\!\!\!\!\!\!\!\!\!
\det\bigg( I(L_s,j_r)\bigg)_{r,s} \!\!\!\!\!\!\det\bigg( I(M_s,j_r)\bigg)_{r,s}
\\ &=& 
\det\Bigg( 2\sum_{j=1}^\infty \left(\frac{2(-1)^{j+1}}{\pi(2j-1)}\right)^{2n-m+1} I(L_s,j)
I(M_r,j) \Bigg)_{r,s}. \label{eq:combined}
\end{eqnarray}
A simple computation, which we include in Appendix A, shows that
\begin{eqnarray} \nonumber
 I(a,j) &=& (-1)^{j+1} \sum_{p=0}^{\lfloor a/2\rfloor} \frac{(-1)^p}{(a-2p)!} \left(\frac{2}{\pi(2j-1)}\right)^{2p+1}
\\ & & + (-1)^{\frac{a+1}{2}} \left(\frac{2}{\pi(2j-1)}\right)^{a+1} 1_{[a\text{ odd}]}. \label{eq:tedious}
\end{eqnarray}
We also verify in Appendix A that
\begin{equation}\label{eq:verify}
 X_N(p,q) = 2\sum_{j=1}^\infty \left(\frac{2 (-1)^{j+1}}{\pi(2j-1)}\right)^N I(p,j) I(q,j).
\end{equation}
Therefore the expression written inside the determinant in \eqref{eq:combined} is exactly
$X_{2n-m+1}(L_s,M_r)$. This proves \eqref{eq:ifeven}.
\end{proof}

\begin{proof}[Sketch of proof of Theorem \ref{combmeaning}]
Eq. \eqref{eq:alphan} follows immediately from the case $m=2$ of Theorem \ref{generalthm}. Eq. \eqref{eq:betan} does not follow from Theorem \ref{generalthm}, but can be proved using the same techniques. In fact, a both formulas can be treated simultaneously by showing that
$$ X_N(p,q) = \lang T^N(x^p/p!),x^q/q! \rang, $$
where $T$ is Elkies's transfer operator defined in \eqref{eq:elkiesop}, and relating this quantity to $\alpha_n$ and $\beta_n$ using Proposition \ref{propmarkov} and the trivial ($k=1$) case of Proposition \ref{propschur}.
\end{proof}

\section{Additional comments and questions \label{sectioncomments}}

We conclude with some final comments and open questions.

\paragraph{1. More enumeration formulas.}
Our treatment of enumeration formulas that can be derived by using the explicit diagonalization of the transfer operators in Theorems \ref{thmdiag2k}, \ref{thmdiag2kp1} is by no means complete and is meant more as an illustration of the power and generality of the technique. One may consider several variant formulas. For example, one may lean the head Young diagram against the vertical side of the triangle (the shape on the right-hand side of Figure \ref{figbelow} is a simple example of this). This will result in enumeration formulas which are minor variations on the formulas of Theorem \ref{generalthm}. More generally, one may consider $m$-strip diagrams with head and tail shapes which are not Young diagrams, so that the overall shape is not necessarily even a skew Young diagram. Formula \eqref{eq:aitkeneq} will no longer apply, but the transfer operator technique will still enable deriving enumeration formulas for such ``generalized tableaux'', although the use of Proposition \ref{propschur} would have to be replaced with more tedious manual computation.

\paragraph{2. The principal eigenfunction.} It is interesting to look at the principal eigenfunction of the transfer operators, since, for example, it controls the limiting behavior of uniform random points in the order polytope of a very long $m$-strip diagram (when $m$ is fixed and $n$ goes to infinity). One gets a simple product representation. For example, for the even case where $m=2k$, the principal eigenfunction is
$$ \psi(x_1,\ldots,x_k) =2^{k/2} \det\Bigg( \cos\left(\frac{\pi (2i-1)x_j}{2}\right)\Bigg)_{1\le i,j\le k}. $$
By representing each of the cosines as a Chebyshev polynomial in $\cos(\pi x_j/2)$, one can transform this determinant into
$$ \psi({\bf x}) =2^{k/2} \prod_{i=1}^k \cos\left(\frac{\pi x_i}{2}\right) 
\prod_{1\le i<j\le k} \Bigg(\cos^2\left(\frac{\pi x_j}{2}\right)-\cos^2\left(\frac{\pi x_i}{2}\right)\Bigg). $$
The probability density $\psi({\bf x})^2$, which will arise as the stationary distribution of the coordinates of a uniform point in the order polytope when cutting along successive diagonals, is similar to eigenvalue densities arising in random matrix theory. One may hope to exploit this to derive interesting results about random $m$-strip tableaux.

\paragraph{3. Open problem: Generating functions.}
It would be interesting to find formulas for the generating functions of some of the families of diagrams treated above. In some cases it's easy. For example, for the diagram $D_n$ in eq. \eqref{eq:form3strip3} the generating function can be computed to be
$$ \sum_{n=0}^\infty \frac{\dd(D_n) x^{2n}}{(3n)!} = x\left(\cot \frac{x}{2} - \cot x\right). $$
In more complicated cases we have not found formulas for the generating functions.

\paragraph{4. Open problem: Bijective and combinatorial proofs.}
Another direction which might be interesting to pursue is to look for combinatorial or bijective proofs of some of our formulas, for example to prove Theorem \ref{thm3strip} by directly relating $3$-strip tableaux to up-down permutations in some combinatorial way.

\paragraph{5. Extension to periodic ribbon tableaux.} 
We restricted our computations to the simplest case of ``periodic''
Young diagrams, those having the slope $(1,1)$. Most of the results
above can be carried over (at the cost of a sizable increase of
complexity) to the case of Young diagrams shaped as stacks of $m$
periodic ribbon Yound diagrams (if one considers our $m$-strip
diagonal YD's as stacks of $m$ {\em up-down} YD's). 
We plan to present these results in a subsequent publication.

\section*{Appendix A. Some computations}

\subsection*{A.1. Proof of \eqref{eq:tedious}}
Let $\theta = \frac{\pi}{2}(2j-1)$. By repeated integration by parts, we get
\begin{eqnarray*}
I(a,j) &=& \frac{1}{a!}\int_0^1 x^a \cos\left(\theta x\right) dx \\
&=& \frac{\theta^{-1}}{a!} x^a\sin\left(\theta x\right)\big|_0^1+\frac{\theta^{-2}}{(a-1)!} x^{a-1}\cos\left(\theta x\right)\big|_0^1 \\
& & -\frac{\theta^{-3}}{(a-2)!}x^{a-2}\sin(\theta x)\big|_0^1-\frac{\theta^{-4}}{(a-3)!}x^{a-3}\cos(\theta x)\big|_0^1\\
 & & +\frac{\theta^{-5}}{(a-2)!}x^{a-4}\sin(\theta x)\big|_0^1+\frac{\theta^{-6}}{(a-5)!}x^{a-5}\cos(\theta x)\big|_0^1 \\
 & & -\frac{\theta^{-7}}{(a-6)!}x^{a-6}\sin(\theta x)\big|_0^1-\frac{\theta^{-8}}{(a-7)!}x^{a-7}\cos(\theta x)\big|_0^1  + \ldots \\
 & = &
 (-1)^{j+1} \sum_{p=0}^{\lfloor a/2\rfloor} \frac{(-1)^p}{(a-2p)!} \theta^{-(2p+1)}
 + (-1)^{\frac{a+1}{2}} \theta^{-(a+1) } 1_{[a\text{ odd}]}.
\end{eqnarray*}
\qed

\subsection*{A.2. Proof of \eqref{eq:verify}}

Rewrite \eqref{eq:equivalent}  as
$$ \bar{A}_n = 2\sum_{\ell=1}^\infty \left(\frac{2(-1)^{\ell+1}}{\pi(2\ell-1)}\right)^{n+1}. $$
Using this and \eqref{eq:tedious}, the right-hand side of \eqref{eq:verify} is seen to be
\begin{eqnarray*}
 &&  \sum_{i=0}^{\lfloor p/2\rfloor}
\sum_{j=0}^{\lfloor q/2\rfloor} \frac{(-1)^{i+j}}{(p-2i)!(q-2j)!} \Bigg[2\sum_{\ell=1}^\infty  
\left(\frac{2(-1)^{\ell+1}}{\pi(2\ell-1)}\right)^{N+2i+2j+2}
\Bigg]\\
& & + (-1)^{\frac{p+1}{2}}\sum_{j=0}^{\lfloor q/2\rfloor} \frac{(-1)^j }{(q-2j)!}
\Bigg[ 2\sum_{\ell=1}^\infty \left(\frac{2(-1)^{\ell+1}}{\pi(2\ell-1)}\right)^{N+p+2j+2}
\Bigg]1_{[p\text{ odd}]}   \\  & & +
(-1)^{\frac{q+1}{2}}\sum_{i=0}^{\lfloor p/2\rfloor} \frac{(-1)^i}{(p-2i)!}
\Bigg[ 2\sum_{\ell=1}^\infty \left(\frac{2(-1)^{\ell+1}}{\pi(2\ell-1)}\right)^{N+2i+q+2}
\Bigg]1_{[q\text{ odd}]}
\\ & & + (-1)^{\frac{p+q}{2}+1}
\Bigg[  2\sum_{\ell=1}^\infty \left(\frac{2(-1)^{\ell+1}}{\pi(2\ell-1)}\right)^{N+p+q+2}
\Bigg] 1_{[p,q\text{ odd}]} \\ 
 &=&  \sum_{i=0}^{\lfloor p/2\rfloor}
\sum_{j=0}^{\lfloor q/2\rfloor} \frac{(-1)^{i+j}\bar{A}_{N+2i+2j+1}}{(p-2i)!(q-2j)!} \\
& & + (-1)^{\frac{p+1}{2}}\sum_{j=0}^{\lfloor q/2\rfloor} \frac{(-1)^j \bar{A}_{N+p+2j+1}}{(q-2j)!}
1_{[p\text{ odd}]}   \\  & & +
(-1)^{\frac{q+1}{2}}\sum_{i=0}^{\lfloor p/2\rfloor} \frac{(-1)^i \bar{A}_{N+q+2i+1}}{(p-2i)!}
1_{[q\text{ odd}]}
\\ & & + (-1)^{\frac{p+q}{2}+1} \bar{A}_{N+p+q+1} 1_{[p,q\text{ odd}]} \\ &=& X_N(p,q).
\end{eqnarray*}
\qed

\end{document}